\documentclass[11pt,reqno]{amsart}
\usepackage{amsmath,amssymb,amscd,latexsym,amsthm,mathrsfs}
\usepackage[usenames,dvipsnames]{color}
\usepackage[unicode]{hyperref}
\usepackage{hypbmsec,longtable}
\numberwithin{equation}{section}
\usepackage{fancyhdr}
\usepackage{graphicx}

\ifx\pdfoutput\undefined\else
\hypersetup{
colorlinks=true,
linkcolor=blue,
citecolor=blue,
urlcolor=blue,
filecolor=blue,
bookmarksnumbered=true,
pdfstartview=FitH,
pdfhighlight=/N
}
\fi

\newtheorem{theorem}{Theorem}
\newtheorem{lemma}{Lemma}

\newtheorem{conjecture}{Conjecture}
\newtheorem{proposition}{Proposition}
\theoremstyle{remark}

\newtheorem*{acknowledgements}{Acknowledgements}

\renewcommand{\Re}{\operatorname{Re}}

\newcommand\PSL{\operatorname{PSL}}
\newcommand\SL{\operatorname{SL}}
\newcommand\id{\operatorname{I}}

\begin{document}

\hypersetup{pdfauthor={Matilde Lal\'in, Francis Rodrigue, Mathew Rogers},%
pdftitle={Secant Zeta Function}}

\title{Secant Zeta Functions}

\author{Matilde Lal\'in}
\address{Department of Mathematics and Statistics, University of Montreal, Montreal, Canada}
\email{mlalin@dms.umontreal.ca}

\author{Francis Rodrigue}
\address{Department of Mathematics and Statistics, University of Montreal, Montreal, Canada}
\email{rodriguefrancis@gmail.com}

\author{Mathew Rogers}
\address{Department of Mathematics and Statistics, University of Montreal, Montreal, Canada}
\email{mathewrogers@gmail.com}

\thanks{This work has been partially supported by NSERC Discovery Grant 355412-2008 and FQRNT Subvention \'etablissement
de nouveaux chercheurs 144987. The work of FR has also been supported by a Bourse d'\'et\'e de premier cycle du ISM-CRM}

\date{April 5, 2013}


\begin{abstract}
We study the series $\psi_s(z):=\sum_{n=1}^{\infty} \sec(n\pi z)n^{-s}$, and prove that it converges under mild restrictions on $z$ and $s$. The function possesses a modular transformation property, which allows us to evaluate $\psi_{s}(z)$ explicitly at certain quadratic irrational values of $z$.  This supports our conjecture that $\pi^{-k} \psi_{k}(\sqrt{j})\in\mathbb{Q}$ whenever $k$ and $j$ are positive integers with $k$ even.  We conclude with some speculations on Bernoulli numbers.
\end{abstract}

\maketitle
\section{Introduction}

Let $\zeta(s)$ denote the Riemann zeta function. It is well known that $\zeta(2n)\pi^{-2n}\in\mathbb{Q}$ for $n\ge 1$.  Dirichlet $L$-functions and Clausen functions are modified versions of the Riemann zeta function, which also have nice properties at integer points \cite{L}.  Berndt studied a third interesting modification of the Riemann zeta function, namely the cotangent zeta function \cite{B}:
\begin{equation}\label{cotangent zeta}
\xi_s (z):= \sum_{n=1}^{\infty} \frac{\cot(\pi n z)}{n^{s}}.
\end{equation}
He proved that \eqref{cotangent zeta} converges under mild restrictions on $z$ and $s$, and he produced many explicit formulas for $\xi_k(z)$, when $z$ is a quadratic irrational, and $k\ge3$ is an odd integer. Consider the following examples:
\begin{align*}
\xi_3\left(\frac{1+\sqrt{5}}{2}\right) = - \frac{ \pi^3}{45\sqrt{5}}, &&\xi_{5}(\sqrt{2})=\frac{\pi^5}{945\sqrt{2}}.
\end{align*}
Berndt's work implies that $\sqrt{j}~\xi_k(\sqrt{j})\pi^{-k}\in\mathbb{Q}$ whenever $j$ is a positive integer that is not a perfect square, and $k\ge3$ is odd.
A natural extension of that work is to replace $\cot(z)$ with one of the functions $\{\tan(z), ~\csc(z), ~\sec(z)\}$.  We can settle the tangent and cosecant cases via elementary trigonometric identities:
\begin{align*}
\sum_{n=1}^{\infty}\frac{\tan(\pi n z)}{n^s}=\xi_s(z)-2\xi_s(2z),&&
\sum_{n=1}^{\infty}\frac{\csc(\pi n z)}{n^s}=\xi_s(z/2)-\xi_s(z),
\end{align*}
but it is more challenging to understand the secant zeta function:
\begin{equation}\label{def}
\psi_{s}(z) := \sum_{n=1}^\infty  \frac{\sec(\pi n z)}{n^{s}}.
\end{equation}
The main goal of this  paper is to prove formulas for specials values of $\psi_s(z)$.  In Section \ref{Section:convergence} we prove that the sum converges absolutely if $z$ is an irrational algebraic number and $s\ge 2$.  
In Section \ref{special values} we obtain results such as
\begin{align*}
\psi_{2}(\sqrt{2})=-\frac{\pi^2}{3},&&\psi_{2}(\sqrt{6})=\frac{2\pi^2}{3}.
\end{align*}
These types of formulas exist because $\psi_{k}(z)$ obeys a modular transformation which we prove in Section \ref{modulartransf} (see equation \eqref{psi functional equation}).  Furthermore, based on numerical experiments, we conjecture:
\begin{conjecture}\label{conjecture on rationality}  Assume that $k$ and $j$ are positive integers, and that $k$ is even.  Then $\psi_{k}(\sqrt{j})\pi^{-k}\in\mathbb{Q}$.
\end{conjecture}
The results of Section \ref{special values} support this conjecture, even though there are still technical hurdles to constructing a complete proof.  For instance, we prove that the conjecture holds for infinite subsequences of natural numbers.  The rational numbers that appear are also interesting, and we speculate on their properties in the conclusion.

\section{Convergence}\label{Section:convergence}
Since $\sec(\pi z)$ has poles at the half-integers, it follows that $\psi_{s}(z)$ is only well-defined if $n z\not\in
\mathbb{Z}+\frac{1}{2}$ for any integer $n$.  Thus, we exclude rational points with even denominators from the domain of $\psi_{s}(z)$.  If $z=p/q$ with $q$ odd, then $\psi_s(p/q)$ reduces to linear combinations of Hurwitz zeta functions, and \eqref{def} converges for $s>1$.  Convergence questions become more complicated if $z$ is irrational.  Irrationality guarantees that $|\sec(\pi n z)|\ne \infty$, but we still have to account for how often $|\sec(\pi n z)|$ is large compared to $n^s$.  The Thue-Siegel-Roth Theorem gives that $|\sec(\pi n z)|\ll n^{1+\varepsilon}$ when $z$ is algebraic and irrational, and this proves that \eqref{def} converges for $s>2$.  The case when $s=2$ requires a more subtle argument.  We use a theorem of Worley to show that the set of $n$'s where $|\sec(\pi n z)|$ is large is sparse enough to ensure that \eqref{def} converges.  We are grateful to Florian Luca for providing this part of the proof.  In summary, we have the following theorem:

\begin{theorem}\label{convergence}
The series in \eqref{def} converges absolutely in the following cases:
\begin{enumerate}
 \item When $z=p/q$ with $q$ odd and $s> 1$.
 \item When $z$ is algebraic irrational, and $s>2$.
 \item When $z$ is algebraic irrational, and $s=2$.
\end{enumerate}
\end{theorem}
\begin{proof}[Proof of Theorem \ref{convergence}, parts (1) and (2)] Let $z$ be a rational number with odd denominator in reduced form. It is easy to see that the set of real numbers
$\{\sec(n\pi z)\}_{n \in \mathbb{N}}$ is finite.  Let $M=\max_{n \in \mathbb{N}} |\sec(n\pi z)|$. Then we have
\begin{equation*}
\frac{|\sec(\pi n z)|}{n^{s}}\leq \frac{M}{n^{s}}.
\end{equation*}
It follows easily from the Weierstrass $M$-test that \eqref{def} converges absolutely for $s>1$.

Now we prove the second part of the theorem.  By elementary estimates
\begin{equation}\label{secant estimate}
|\sec(\pi n z)|=\left|\csc\left(\pi(n z-1/2)\right)\right|\ll \frac{1}{\left|n z-\frac{1}{2}-k_n\right|},
\end{equation}
where $k_n$ is the integer which minimizes $|n z-\frac{1}{2}-k_n|$.  Now appeal to the Thue-Siegel-Roth Theorem \cite{R}.  In particular, for any algebraic irrational number $\alpha$, and given $\varepsilon>0$, there
exists a constant $C(\alpha, \varepsilon)$, such that
\begin{equation}\label{Roth's theorem}
\left|\alpha-\frac{p}{q}\right|> \frac{C(\alpha, \varepsilon)}{q^{2+\varepsilon}}.
\end{equation}
If we set $\alpha=z$, then \eqref{secant estimate} becomes
\begin{equation*}
|\sec(\pi n z)|\ll \frac{1}{n\left|z-\frac{2k_n+1}{2n}\right|}\ll n^{1+\varepsilon}.
\end{equation*}
Therefore we have
\begin{equation*}
\frac{|\sec(n\pi z)|}{n^{s}}\ll \frac{1}{n^{s-1-\varepsilon}},
\end{equation*}
and this implies that \eqref{def} converges absolutely for $s>2+\varepsilon$.  Since $\varepsilon$ is arbitrarily small the result follows.
\end{proof}

In order to prove the third part of Theorem \ref{convergence}, we require some background on continued fractions.  Recall that any irrational number $z$ can be represented as an infinite continued fraction
\[z=[a_0;a_1,a_2,\dots],\]
and the convergents are given by
\[[a_0;a_1,\dots, a_\ell]=\frac{p_\ell}{q_\ell},\]
which satisfy
\begin{eqnarray}
\label{rec1} p_\ell&=&a_{\ell} p_{\ell-1}+p_{\ell-2},\\
\label{rec2} q_\ell&=&a_{\ell} q_{\ell-1}+q_{\ell-2}.
\end{eqnarray}
Convergents provide the best possible approximations to algebraic numbers among rational numbers with bounded denominators.  In other words, if $0<q<q_\ell$, then
\begin{equation}\label{convergent1}
\left| z -\frac{p}{q}\right|>\left| z -\frac{p_\ell}{q_\ell}\right|.
\end{equation}
In addition
\begin{equation}\label{convergent2}
\frac{1}{q_\ell q_{\ell +1}}>\left| z -\frac{p_\ell}{q_\ell}\right| > \frac{1}{q_\ell(q_{\ell+1}+q_\ell)}.
 \end{equation}
Now we state a weak version of a theorem due to Worley \cite[Thm. 1]{Wor81}.

\begin{theorem}[Worley]\label{Worley}
Let $z$ be irrational, $k \geq \frac{1}{2}$, and $p/q$ be a rational approximation to $z$ in reduced form for which
\[\left| z - \frac{p}{q} \right| <\frac{k}{q^2}.\]
Then either $p/q$ is a convergent $p_\ell/q_\ell$ to $z$, or
\[\frac{p}{q} = \frac{a p_\ell+ b p_{\ell-1}}{aq_\ell+bq_{\ell-1}}, \quad |a|,|b| < 2k,\]
where $a$ and $b$ are integers.
\end{theorem}
Now we can complete the proof of Theorem \ref{convergence}.  The following proof was kindly provided by Florian Luca.

\begin{proof}[Proof of Theorem \ref{convergence}, part (3)]
Let $k_n$ be the integer which minimizes $|n z-\frac{1}{2}-k_n|$.  Let $W_z$ denote the set of integers where the quantity is large:
\begin{equation*}
W_z =\left \{ n\in \mathbb{N} \, :\, \left|n z-\frac{1}{2}-k_n\right|\geq\frac{(\log n)^2}{n} \right\}.
\end{equation*}
Then
\begin{equation*}
\sum_{n \in W_z} \frac{|\sec(n \pi z)|}{n^2} \ll |\sec(\pi z)|+ \sum_{n=2}^\infty \frac{1}{n (\log n)^2},
\end{equation*}
which converges.

Now assume that $n \not \in W_z$.  Then
\begin{equation*}
 \left|z-\frac{1+2k_n}{2n}\right|<\frac{(\log n)^2}{n^2}.
\end{equation*}
Consider the convergents of $z$. Let $\ell$ be such that $q_{\ell-1} \leq 2 n < q_{\ell}$.
By Theorem \ref{Worley} there are at most  $O\left((\log q_{\ell})^4\right)$ solutions to
\begin{equation}\label{notW}
 \left|z-\frac{p}{2n}\right|<\frac{(\log n)^2}{n^2}
 \end{equation}
with $p \in \mathbb{Z}$ (i.e. consider all values of $|a|,|b|<2k=2(\log q_\ell)^2$).

From equations \eqref{convergent1} and \eqref{convergent2} we have
\[ \left| z-\frac{p}{2 n}\right|\geq \left|z-\frac{p_{\ell}}{q_{\ell}}\right|\geq \frac{1}{q_{\ell}(q_{\ell+1}+q_{\ell})}.\]
Combining this with equation \eqref{rec2} implies
\[ \left|nz-\frac{1+2k_n}{2}\right|n^2\geq \frac{n^3}{q_\ell(q_{\ell+1}+q_{\ell})} \geq \frac{q_{\ell-1}^3}{8q_\ell\left(q_{\ell}(a_{\ell+1}+1)+q_{\ell-1}\right)}\geq \frac{q_{\ell-1}^3}{8q_\ell^2(a_{\ell+1}+2)}.\]
Hence, if $n \not\in W_z$, we find that
\[\frac{|\sec(n \pi z)|}{n^2} \ll \frac{q_{\ell}^2(a_{\ell+1}+2)}{q_{\ell-1}^3}\ll \frac{a_{\ell+1}q_\ell^2}{q_{\ell-1}^3}.\]
Combining the Thue-Siegel-Roth Theorem (equation \eqref{Roth's theorem}) with \eqref{convergent2}, implies that if $z$ is algebraic
\[\frac{1}{q_\ell q_{\ell+1}}>\left|z -\frac{p_\ell}{q_\ell}\right| > \frac{C(z,\varepsilon)}{q_\ell^{2+\varepsilon}}.\]
Thus $q_{\ell+1}\ll  q_\ell^{1+\varepsilon}$.  This allows us to place an upper bound on $a_{\ell+1}$:
\[a_{\ell+1} \leq  \frac{q_{\ell+1}}{q_\ell}\ll q_\ell^\varepsilon\]
Putting everything together gives the bound
\[\frac{|\sec(n \pi z)|}{n^2} \ll \frac{1}{q_{\ell-1}^{1-4\varepsilon}},\]
and as a result
\[\sum_{n \not \in W_z} \frac{|\sec(n \pi z)|}{n^2} \ll \sum_{\ell=1}^\infty \frac{(\log q_\ell)^4}{q_{\ell}^{1-\varepsilon'}} \ll \sum_{\ell=1}^\infty \frac{1}{q_{\ell}^{1-\varepsilon''}}.\]
Since $q_{\ell+1}=a_{\ell+1}q_\ell+q_{\ell-1} \geq q_\ell+q_{\ell-1}$, we conclude that $q_\ell \geq F_\ell$, where $F_\ell$ denotes the $\ell$th Fibonacci number.  Since the Fibonacci numbers grow exponentially, we have
\[q_\ell \gg \frac{\varphi^\ell}{\sqrt{5}}, \mbox{ where } \varphi=\frac{1+\sqrt{5}}{2}.\]
Setting $\tilde{\varphi}=\varphi^{1-\varepsilon''}>1$, we finally obtain
\[\sum_{n \not \in W_z} \frac{|\sec(n \pi z)|}{n^2} \ll \sum_{\ell=1}^\infty \frac{1}{\tilde{\varphi}^\ell}=\frac{1}{\tilde{\varphi}-1} <\infty.\]
Thus, it follows that \eqref{def} converges absolutely when $s=2$.
\end{proof}

\section{A modular transformation for $\psi_{k}(z)$}\label{modulartransf}
We begin by noting the following trivial properties of $\psi_s(z)$:
\begin{align}
\psi_{s}(-z)=&\psi_{s}(z),\label{psi even}\\
\psi_{s}(z+2)=&\psi_{s}(z),\label{psi period}\\
2^{1-s}\psi_{s}(2z)=&\psi_{s}(z)+\psi_{s}(z+1).\label{psi semi period}
\end{align}
The main goal of this section is to prove that $\psi_{k}(z)$ also satisfies a modular transformation formula.  We start from the partial fractions decompositions of $\sec(\pi x )$ and $\csc(\pi x)$, and then perform a convolution trick to obtain an expansion for $\sec(\pi x)\csc(\pi x z)$ (equation \eqref{lemma formula}).  Differentiating with respect to $x$ then leads to the transformation for $\psi_{k}(z)$ (equation \eqref{psi functional equation}).
This method is originally due to the third author, who used used it to rediscover the Newberger summation rule for Bessel functions \cite{Rg}, \cite{Ne}:
\begin{equation}\label{newberger}
J_{x}(y)J_{-x}(y)\frac{\pi}{\sin(\pi x)}=\sum_{n=-\infty}^{\infty}\frac{J_{n}^2(y)}{n+x}.
\end{equation}
Equation \eqref{newberger} follows from applying the convolution trick to partial fractions expansions for $J_x(y)$ and $J_{-x}(y)$ \cite{Lo}.

\begin{lemma} \label{lemma trick} Let $\chi_{-4}(n)$ denote the Legendre symbol modulo $4$.  Suppose that $x$ and $z$ are selected appropriately.\footnote{We assume that $z$ is irrational and algebraic, and $x$ is selected so that the denominators in the sums are never zero.} Then
\begin{equation}\label{lemma formula}
\begin{split}
\pi\csc(\pi z x)\sec(\pi x )=&\frac{1}{z x}+8 x\sum_{n=1}^{\infty}\frac{\chi_{-4}(n) \csc\left(\pi n z /2\right)}{n^2-4 x^2}
-2z x\sum_{n=1}^{\infty}\frac{\sec\left(\pi n (1+1/z)\right)}{n^2-z^2 x^2}.
\end{split}
\end{equation}
\end{lemma}
\begin{proof}Recall the classical partial fractions expansions \cite{Gr}:
\begin{align}
\pi\left(\sec(\pi x)-1\right)=&16 x^2 \sum _{n=1}^{\infty } \frac{\chi_{-4}(n)}{n \left(n^2-4x^2\right)},\label{trig 1}\\
\pi\csc(\pi x)-\frac{1}{x}=&
2x\sum_{k=1}^{\infty}\frac{(-1)^{k+1}}{k^2-x^2}.\label{trig 2}
\end{align}
Both sums converge uniformly.  Multiplying the formulas together, expanding via partial fractions, and rearranging the order of summation, we have
\begin{align*}
\pi\left(\sec(\pi x )-1\right)\left(\pi\csc(\pi z x)-\frac{1}{z x}\right)
=&32z x^3\sum_{\substack{n\ge1\\k\ge 1}}\frac{(-1)^{k+1}\chi_{-4}(n)}{n\left(n^2-4 x^2\right)\left(k^2-z^2 x^2\right)}\\
=&8z^3 x^3\sum_{\substack{n\ge1\\k\ge 1}}\frac{(-1)^{k+1}\chi_{-4}(n)}{n\left(k^2-z^2 n^2/4\right)}\left(\frac{1}{z^2 n^2/4-z^2 x^2}-\frac{1}{k^2-z^2 x^2}\right)\\
=&32z x^3\sum_{n=1}^{\infty}\frac{\chi_{-4}(n)}{n\left(n^2-4 x^2\right)}\left(\sum_{k=1}^{\infty}\frac{(-1)^{k+1}}{k^2-z^2 n^2/4}\right)\\
&+32z x^3\sum_{k=1}^{\infty}\frac{(-1)^{k+1}}{k^2-z^2x^2}\left(\sum_{n=1}^{\infty}\frac{\chi_{-4}(n)}{n\left(n^2-4 k^2/z^2\right)}\right).
\end{align*}
By \eqref{trig 1} and \eqref{trig 2} this becomes
\begin{align*}
\pi\left(\sec(\pi x )-1\right)\left(\pi\csc(\pi z  x )-\frac{1}{z x}\right)=&32x^3\sum_{n=1}^{\infty}\frac{\chi_{-4}(n)}{n^2\left(n^2-4 x^2\right)}\left(\pi\csc\left(\pi n z/2\right)-\frac{2}{n z}\right)\\
&+2\pi z^3 x^3\sum_{k=1}^{\infty}\frac{(-1)^{k+1}}{k^2(k^2-z^2 x^2)}\left(\sec(\pi k/z)-1\right).
\end{align*}
This reduces to \eqref{lemma formula} after several additional applications of \eqref{trig 1} and \eqref{trig 2}.  We can split up the sums because all of the individual terms converge absolutely.  For instance, we can prove that
\begin{align*}
\sum_{k=1}^{\infty}\frac{(-1)^{k+1}\sec(\pi k/z)}{k^2-z^2 x^2},
\end{align*}
converges absolutely, by showing that the summand is $\ll |\sec(\pi k/z)|k^{-2}$, and then applying Theorem \ref{convergence} for appropriate choices of $z$.
\end{proof}

\begin{theorem}\label{theorem 1}Let $E_m$ denote the Euler numbers, and let $B_m$ denote the Bernoulli numbers. Suppose that $k\in2\mathbb{N}$.  Then for appropriate choices of $z$:
\begin{equation}\label{psi functional equation}
\begin{split}
(1+z)^{k-1} \psi_{k}&\left(\frac{z}{1+z}\right) - (1-z)^{k-1} \psi_{k} \left(\frac{z}{1-z}\right)\\
&= \frac{(\pi i)^{k}}{k!} \sum_{m=0}^{k}(2^{m-1}-1) B_{m} E_{k-m} \binom{k}{m}\left[(1+z)^{m-1}-(1-z)^{m-1}\right].
\end{split}
\end{equation}
\end{theorem}
\begin{proof} Recall the Taylor series expansions of cosecant and secant:
\begin{align*}
\pi x \csc(\pi x)=-\sum_{m=0}^\infty (2^m-2)B_m \frac{(\pi i x)^m}{m!},&&
\sec(\pi x)=\sum_{m=0}^{\infty}E_m\frac{(\pi i x)^m}{m!}.
\end{align*}
Expand both sides of \eqref{lemma formula} in a Taylor series with respect to $x$.  Comparing coefficients yields the following identity:
\begin{equation*}
\frac{(\pi i)^k}{k!}\sum_{m=0}^{k}(2^{m-1}-1)B_m E_{k-m}{k\choose m}z^{m-1}=-2^{k}\sum_{n=1}^{\infty}\frac{\chi_{-4}(n)\csc(\pi n z/2)}{n^{k}}+ z^{k-1}\psi_{k}\left(1+1/z\right).
\end{equation*}
Finally, let $z\mapsto(1+z)$ and $z\mapsto(1-z)$, and subtract the two results to recover \eqref{psi functional equation}.  The cosecant sums vanish because $\csc(\pi n (1+z)/2)=\csc(\pi n (1-z)/2)$ whenever $n$ is odd.
\end{proof}
We conclude this subsection with a conjecture on unimodular polynomials.  We call a polynomial \textit{unimodular} if its zeros all lie on the unit circle.  We have observed numerically that the polynomials in \eqref{psi functional equation} have all of their zeros on the vertical line $\Re(z)=0$.  Since the linear fractional transformation $z=(1-x)/(1+x)$ maps the vertical line to the unit circle, we arrive at the following conjecture:
\begin{conjecture}\label{conjecture 1} We conjecture that the polynomial
\begin{equation}
\sum_{m=0}^{k}2^m (2^m-2)B_m E_{k-m}{k\choose m}(x-x^m)(1+x)^{k-m},
\end{equation}
is unimodular when $k\in2\mathbb{N}$.
\end{conjecture}
This new family of polynomials is closely related to the unimodular polynomials introduced in \cite{LR}, \cite{LS}, and \cite{WZ}.

\section{Special values of $\psi_{k}(z)$}\label{special values}

Throughout this section we assume that $k$ is a positive even integer.  Let $\SL_{2}(\mathbb{Z})$ denote the set of $2\times2$ integer-valued matrices with determinant equal to one, and let $\PSL_{2}(z)=\SL_{2}(\mathbb{Z})/\langle \id,-\id\rangle$.  Recall that $\PSL_2(\mathbb{Z})$ can be identified with the set of linear fractional transformations.  The usual group action is
\begin{equation*}
\left( \begin{array}{cc}a&b\\c&d\end{array}\right)z = \frac{az+b}{cz+d}.
\end{equation*}
It is easy to see that multiplying matrices is equivalent to performing compositions of linear fractional transformations.  Consider the following matrices in $\PSL_{2}(\mathbb{Z})$:
\begin{align*}
A=\left( \begin{array}{cc}1&2\\0&1\end{array}\right),&&B=\left( \begin{array}{cc}1&0\\2&1\end{array}\right).
\end{align*}
Equations \eqref{psi period} and \eqref{psi functional equation} are equivalent to
\begin{align}
\psi_{k}\left(A z\right)=&\psi_{k}(z),\label{matrix A}\\
\psi_{k}\left(B z\right)=&(2z+1)^{1-k} \psi_k(z)\label{matrix B}\\
&+\frac{(\pi  i)^k }{k!}\sum _{m=0}^k \left(2^{m-1}-1\right) B_m E_{k-m} {k\choose m}(z+1)^{k-m} \left[(2 z+1)^{m-k}-(2 z+1)^{1-k}\right].\notag
\end{align}
Every matrix $C\in \langle A,B\rangle$ has a factorization of the form $C=A^{j_1}B^{j_2}A^{j_3}\dots$, so equations \eqref{matrix A} and \eqref{matrix B} together imply that there exists a $z$-linear relation between $\psi_{k}(z)$, $\psi_k(C z)$, and $\pi^k$.  

Now we outline a strategy to obtain exact evaluations of $\psi_{k}(z)$.  First select a matrix $C$, and then find the linear relation between $\psi_{k}(z)$, $\psi_{k}(C z)$, and $\pi^k$.  Next choose $z$ so that $\psi_{k}(z)=\psi_{k}(C z)$.  For example, if $z=2j+\sqrt{2j(2j+1)}$ in \eqref{matrix B}, then $Bz=z-4j$ whenever $j$ is a non-zero integer. With some work the equation collapses to
\begin{equation}\label{gen 1}
\begin{split}
\psi_{k}\left(\sqrt{2j(2j+1)}\right)=\frac{(\pi i)^k}{k!}\sum_{m=0}^{k}&\left(2^{m-1}-1\right)B_{m}E_{k-m}{k\choose m}\\
&\times\left[\frac{\left(1+\sqrt{\frac{2j}{2j+1}}\right)^{m-1}-\left(1-\sqrt{\frac{2j}{2j+1}}\right)^{m-1}}
{\left(1+\sqrt{\frac{2j}{2j+1}}\right)^{k-1}-\left(1-\sqrt{\frac{2j}{2j+1}}\right)^{k-1}}\right],
\end{split}
\end{equation}
 which holds for $j\in\mathbb{Z}\setminus\{0\}$.  Similarly, if we take $z=|2j+1|+ \sqrt{2 j (2 j + 1)}$, then we arrive at
\begin{equation}\label{gen 2}
\begin{split}
\psi_{k}\left(1+\sqrt{2j(2j+1)}\right)=\frac{(\pi i)^k}{k!}\sum_{m=0}^{k}&\left(2^{m-1}-1\right)B_{m}E_{k-m}{k\choose m}\\
&\times\left[\frac{\left(1+\sqrt{\frac{2j+1}{2j}}\right)^{m-1}-\left(1-\sqrt{\frac{2j+1}{2j}}\right)^{m-1}}
{\left(1+\sqrt{\frac{2j+1}{2j}}\right)^{k-1}-\left(1-\sqrt{\frac{2j+1}{2j}}\right)^{k-1}}\right],
\end{split}
\end{equation}
which also holds for $j\in\mathbb{Z}\setminus\{0\}$.  The right-hand sides of \eqref{gen 1} and \eqref{gen 2} are invariant under the Galois action $\sqrt{x}\mapsto-\sqrt{x}$, so both formulas are rational with respect to $j$.  Specializing \eqref{gen 1} at $k=2$ and $k=4$ yields
\begin{align*}
\psi_{2}\left(\sqrt{2j(2j+ 1)}\right)=&(3j+1)\frac{\pi^2}{6},\\
\psi_{4}\left(\sqrt{2j(2j+1)}\right)=&\left( \frac{75j^2+ 46 j + 6}{8j+3}\right)\frac{\pi^4}{180}.
\end{align*}
These results support Conjecture \ref{conjecture on rationality}.  Further evidence for the conjecture is provided by combining \eqref{gen 1}, \eqref{gen 2}, and \eqref{psi semi period}, to obtain identities like
\begin{equation}\label{psi2 fixed case}
\psi_2\left(\sqrt{8j(2j+1)}\right)=\frac{\pi^2}{6},
\end{equation}
whenever $j\in \mathbb{Z}$.

In general, it seems to be quite difficult to evaluate $\psi_k(\sqrt{j})$ for arbitrary positive integers $j$.  This is due to the fact that \eqref{matrix A} and \eqref{matrix B} restrict the available matrices to a subgroup of $\PSL_2(\mathbb{Z})$.  If there exists a matrix $C\in\langle A,B\rangle$ which satisfies $C \sqrt{j}=\sqrt{j}$, then we can always evaluate $\psi_{k}(\sqrt{j})$.  We can construct candidate matrices by solving Pell's equation:
\begin{equation*}
X^2-j Y^2=1.
\end{equation*}
If we choose $C$ to be given by
\begin{equation*}
C=\left( \begin{array}{cc}X&j~Y\\Y&X\end{array}\right),
\end{equation*}
then it follows that $C \sqrt{j}=\frac{X\sqrt{j}+j Y}{Y\sqrt{j}+X}=\sqrt{j}$, and $\det(C)=X^2-j Y^2=1$.  Pell's equation has infinitely many integral solutions when $j\ge 1$ \cite{HW}, so there are infinitely many choices of $C$.  The main difficulty is to select appropriate values of $X$ and $Y$ so that $C$ factors into products of $A$'s and $B$'s.  It is not clear if such a selection is always possible.  Notice that $\langle A, B \rangle\subset\Gamma_{0}(2)$ and so it follows that $\PSL_2(\mathbb{Z})\not\subset\langle A, B \rangle$.

We conclude this subsection by noting that $\psi_{2}(z)=0$ for infinitely many irrational values of $z$.  Setting $n=-3j$ in Proposition \ref{proposition with many cases} below, yields
\begin{equation}
\psi_2\left(\sqrt{\frac{2(6j^2-1)}{3}}\right)=0,
\end{equation}
for any non-zero integer $j$.  It is worth emphasizing that $\psi_{k}(z)$ is highly discontinuous with respect to $z$, so these types of results are not surprising.
\begin{proposition}\label{proposition with many cases} Suppose that $j$ and $n$ are integers, and $n\ne 0$. Then
\begin{equation}\label{generalized irrational}
\psi_2 \left(\sqrt{\frac{2j(2j n+1)}{n}} \right)= \left(1+\frac{3j}{n}\right)\frac{\pi^2}{6}.
\end{equation}
\end{proposition}
\begin{proof} Setting $k=2$ in \eqref{matrix B} yields
\begin{equation}\label{special case}
\psi_2(B z)=\frac{1}{2z+1}\psi_2(z)+\frac{z(3z^2 + 4z + 2)}{(2z+1)^2}\frac{\pi^2}{6}.
\end{equation}
Iterating \eqref{special case} gives
\begin{align}
\psi_{2}(B^n z)
=&\frac{1}{2n z+1}\psi_{2}(z)+\frac{n z \left(3z^2+4 n z+2\right)}{(2n z+1)^2}\frac{\pi^2}{6}.\label{B^n action}
\end{align}
The derivation of \eqref{B^n action} is best accomplished with the aid of a computer algebra system such as \texttt{Mathematica}, because significant telescoping occurs on the right.  Now consider the matrix
\begin{equation*}
C=A^j B^n A^j = \left(\begin{array}{cc}4j n+1&4j(2j n+1)\\2n&4j n+1\end{array} \right),
\end{equation*}
and notice that $C z_0 =z_0$, where
\begin{equation*}
z_0=\sqrt{\frac{2j(2j n+1)}{n}}.
\end{equation*}
Thus by \eqref{B^n action} we have
\begin{align*}
\psi_{2}(z_0)
=&\psi_{2}(C z_0)\\
=&\psi_{2}\left(B^n(z_0+2j)\right)\\
=&\frac{1}{2n (z_0+2j)+1}\psi_{2}(z_0)+\frac{n (z_0+2j) \left(3(z_0+2j)^2+4 n (z_0+2j)+2\right)}{(2n (z_0+2j)+1)^2}\frac{\pi^2}{6}.
\end{align*}
We complete the proof by solving for $\psi_{2}(z_0)$ and simplifying.
\end{proof}

\section{Speculations and Conclusion}
Assume that $k$ is a positive even integer.  Euler gave the following expression for Bernoulli numbers:
\begin{equation}\label{Bernoulli definition}
B_{k} = -2\frac{k!}{(2\pi i)^{k}} \zeta(k).
\end{equation}
Bernoulli numbers are interesting combinatorial objects, so it is natural to ask if the rational part of $\psi_{k}(\sqrt{j})$ also has interesting properties.  This is an obvious question because $\psi_{k}(0)=\zeta(k)$.  For instance, the von Staudt-Clausen theorem gives a complete description of the denominators of Bernoulli numbers:
\begin{equation}\label{VonStaudt}
B_{k} = \sum_{(p-1)|k} \frac{1}{p} + \text{ Integer}.
\end{equation}
Is there an analogue of \eqref{VonStaudt} for the rational part of $\psi_{k}(\sqrt{j})$?  As an example, consider $\psi_{k}(\sqrt{6})$ which can be calculated from equation \eqref{gen 1}.  In order to eliminate some trivial factors, we define
\begin{equation*}
\beta_k:=\frac{\left(3+\sqrt{6}\right)^{k-1}-\left(3-\sqrt{6}\right)^{k-1} }{\sqrt{6} }\frac{k!}{ (\pi  i)^k} \psi_k(\sqrt{6}).
\end{equation*}
By \eqref{gen 1} we have
\begin{equation*}
\beta_k= \sum _{m=0}^k \left(2^{m-1}-1\right)B_m E_{k-m} \binom{k}{m} 3^{k-m}\frac{\left(3+\sqrt{6}\right)^{m-1}-\left(3-\sqrt{6}\right)^{m-1} }{\sqrt{6}}.
\end{equation*}
The first few values are $\beta_2=-8/3$, $\beta_4=508/5$, and $\beta_6=-64896/7$, which all have denominators divisible only by primes where $p-1|k$ (as hoped for).  The first instance where this fails is $\beta_{20}$. The denominator of $\beta_{20}$ equals $5\cdot 7\cdot11$, but $7-1=6$ does not divide $20$.

\acknowledgements{The authors wish to thank Shabnam Ahktari for helpful discussions and Florian Luca for providing the proof of part (3) of Theorem \ref{convergence}. Finally, the authors are grateful to the referee for comments and suggestions that greatly improved the exposition of this article.}

\end{document}